\documentclass[english,11pt,reqno]{smfart}

\usepackage{color}
\usepackage{amsthm}
\usepackage{amsmath}
\usepackage{amsfonts}
\usepackage{amstext}
\usepackage[latin1]{inputenc}
\usepackage{amscd}
\usepackage{latexsym}
\usepackage{bm}
\usepackage{amssymb}
\usepackage[all]{xy}
\usepackage{euscript}
\usepackage{a4wide}
\usepackage{amsmath,amssymb,graphicx}
\usepackage{amssymb}
\usepackage{subfig}
\usepackage{amsmath}
\usepackage{mathrsfs,mathtools}

\newtheorem{theorem}{Theorem}
\newtheorem{proposition}[theorem]{Proposition}
\newtheorem{lemma}[theorem]{Lemma}

\newtheorem{definition}[theorem]{Definition}
\newtheorem{remark}[theorem]{Remark}

\def\di{\displaystyle}

\newcommand{\N}{\mathbb{N}}
\newcommand{\R}{\mathbb{R}}
\newcommand{\CC}{\mathscr{C}}

\newcommand{\fonction}[5]{\begin{array}[t]{lrcl}#1 :&#2 &\longrightarrow &#3\\&#4& \longmapsto &#5 \end{array}}

\newcommand{\fonctionsansdef}[3]{\begin{array}[t]{lrcl}#1 :&#2 &\longrightarrow &#3 \end{array}}

\DeclareMathOperator*{\limit}{\longrightarrow}
\begin{document}

\title{\textbf{A continuous/discrete fractional Noether's theorem}}
\author{Lo\"ic Bourdin}
\address{Laboratoire de Math\'ematiques et de leurs Applications - Pau (LMAP). UMR CNRS 5142. Universit\'e de Pau et des Pays de l'Adour.}
\email{bourdin.l@etud.univ-pau.fr}

\author{Jacky Cresson}
\address{Laboratoire de Math\'ematiques et de leurs Applications - Pau (LMAP). UMR CNRS 5142. Universit\'e de Pau et des Pays de l'Adour.}
\email{jacky.cresson@univ-pau.fr}

\author{Isabelle Greff}
\address{Laboratoire de Math\'ematiques et de leurs Applications - Pau (LMAP). UMR CNRS 5142. Universit\'e de Pau et des Pays de l'Adour.}
\email{isabelle.greff@univ-pau.fr}
\date{}
\maketitle

\begin{abstract}
We prove a fractional Noether's theorem for fractional Lagrangian systems invariant under a symmetry group both in the continuous and discrete cases. This provides an explicit conservation law (first integral) given by a closed formula which can be algorithmically implemented. In the discrete case, the conservation law is moreover computable in a finite number of steps.
\end{abstract}



\section{Introduction}
Fractional calculus is the emerging mathematical field dealing with the generalization of the derivative to any real order. During the last two decades, it has been successfully applied to problems in economics \cite{comt}, computational biology \cite{magi} and several fields in Physics \cite{bagl,hilf3,alme,neel}. Fractional differential equations are also considered as an alternative model to non-linear differential equations, see \cite{boni}. We refer to \cite{oldh,samk,podl,kilb} for a general theory and to \cite{mach} for more details concerning the recent history of fractional calculus. Particularly, a subtopic of the fractional calculus has recently gained importance: it concerns the variational principles on functionals involving fractional derivatives. This leads to the statement of fractional Euler-Lagrange equations, see \cite{riew,agra,bale2}. \\

A conservation law for a continuous or a discrete dynamical system is a function which is constant on each solution. Precisely, let us denote by $q = ( q(t) )_{ t \in [a,b]}$ (resp. $Q=(Q_k )_{k=0,\dots ,N}$) the solutions of a continuous (resp. discrete) system taking values in a set denoted by $E$. Then, a function $C : E \rightarrow \R$ is a conservation law (or first integral) if, for any solution $q$ (resp. $Q$), it exists a real constant $c$ such that:
\begin{equation}
C\big( q(t) \big)=c,\; \forall t\in [a,b]. \quad (\text{resp.}\; C(Q_k )=c ,\; \forall k=0,\dots ,N.)
\end{equation} 
Conservation laws are generally associated to some physical quantities like total energy or angular momentum in mechanical systems and sometimes, they also can be used to reduce or integrate by quadrature the equation.\\

In this paper, we study the existence of explicit conservation laws for continuous and discrete fractional Lagrangian systems introduced in \cite{agra} and \cite{bour} respectively. We follow the usual strategy of the classical Noether's theorem providing an explicit conservation law for classical Lagrangian systems invariant under symmetries. Precisely, we prove a fractional Noether's theorem providing an explicit conservation law for fractional Lagrangian systems invariant under symmetries both in the continuous and discrete cases. The given formula is algorithmic so that arbitrary hight order approximations can be computed. Moreover, the algorithm is finite in the discrete case, .\\

Previous results in this direction (in the continuous case) have been obtained by several authors, see \cite{cres6,torr3,atan,musl}. However, in each of these papers, the conservation law is not explicitly derived in terms of symmetry group and Lagrangian but implicitly defined by functional relations (see \cite{torr3,musl}) or integral relations (see \cite{atan}).\\

The paper is organized as follows. In Section \ref{section1}, we first remind classical definitions and results concerning fractional derivatives and fractional Lagrangian systems. Then, we prove a transfer formula leading to the statement of a fractional Noether's theorem based on a result of Cresson in \cite{cres6}. Finally, we give some remarks concerning the previous studies on the subject and we give some precisions about our result. In Section \ref{section2}, we first remind the definition of discrete fractional Lagrangian systems in the sense of \cite{bour}. Then, introducing the notion of discrete symmetry for these systems, we state a discrete fractional Noether's theorem. We conclude with some remarks and a numerical implementation of the discrete fractional harmonic oscillator. 

\section{A fractional Noether's theorem}\label{section1}
In this paper, we consider fractional differential systems in $\R^d$ where $d \in \N^*$ is the dimension ($\N^*$ denotes the set of positive integer). The trajectories of these systems are curves $q$ in $\CC^0 ([a,b],\R^d)$ where $a<b$ are two reals.

\subsection{Reminder about fractional Lagrangian systems}\label{section11}
We first review classical definitions extracted from \cite{samk,podl,kilb} of the fractional derivatives of Riemann-Liouville. Let $\alpha > 0$ and $f$ be a function defined on $[a,b]$ with values in $\R^d$. The left (resp. right) fractional integral in the sense of Riemann-Liouville with inferior limit $a$ (resp. superior limit $b$) of order $ \alpha $ of $f$ is given by:
\begin{equation}
\forall t \in ]a,b], \; I^{\alpha}_{-} f(t) = \dfrac{1}{\Gamma (\alpha)} \di \int_a^t (t-y)^{\alpha -1} f(y) \; dy ,
\end{equation}
respectively:
\begin{equation}
\forall t \in [a,b[, \; I^{\alpha}_{+} f(t) = \dfrac{1}{\Gamma (\alpha)} \di \int_t^b (y-t)^{\alpha -1} f(y) \; dy ,
\end{equation}
where $\Gamma$ denotes the Euler's Gamma function and provided the right side terms are defined. For $\alpha =0$, let us define $I^{0}_{-} f = I^{0}_{+} f = f $. \\

From now, let us consider $ 0 < \alpha \leq 1$. The left (resp. right) fractional derivative in the sense of Riemann-Liouville with inferior limit $a$ (resp. superior limit $b$) of order $\alpha$ of $f$ is given by:
\begin{equation}
\forall t \in ]a,b], \; D^\alpha_- f (t) = \dfrac{d}{dt} \Big( I^{1-\alpha}_- f \Big) (t) ,
\end{equation}
respectively:
\begin{equation}
\forall t \in [a,b[, \; D^\alpha_+ f (t) = -\dfrac{d}{dt} \Big( I^{1-\alpha}_+ f \Big) (t) ,
\end{equation}
provided the right side terms are defined. \\

A fractional Lagrangian functional is an application defined by:
\begin{equation}\label{eqlagfuncfrac}
\fonction{\mathcal{L}^\alpha}{\CC^{2}([a,b],\R^d)}{\R}{q}{\di \int_{a}^{b} L \big( q(t),D^\alpha_- q (t),t \big) \; dt,} 
\end{equation}
where $L$ is a Lagrangian \textit{i.e.} a $\CC^{2}$ application defined by:
\begin{equation}
\fonction{L}{\R^d \times \R^d \times [a,b]}{\R}{(x,v,t)}{L(x,v,t).}
\end{equation}
It is well-known that extremals of a fractional Lagrangian functional can be characterized as solutions of a fractional differential system, see \cite{agra}. Precisely, $q$ is a critical point of $\mathcal{L}^\alpha$ if and only if $q$ is a solution of the fractional Euler-Lagrange equation given by:
\begin{equation}\label{elf}\tag{EL${}^\alpha$}
\dfrac{\partial L}{\partial x} \big( q,D^{\alpha}_{-} q,t \big) + D^{\alpha}_{+} \left( \dfrac{\partial L}{\partial v} \big( q,D^{\alpha}_{-} q,t \big) \right) = 0 . 
\end{equation}
A dynamical system governed by an equation shaped as \eqref{elf} is called fractional Lagrangian system. \\

For $\alpha =1$, $D^1_- = - D^1_+ = d/dt$ and then {\rm (EL${}^1$)} coincides with the classical Euler-Lagrange equation (see \cite{arno}).

\subsection{Symmetries} 

We first review the definition of a one parameter group of diffeomorphisms: 
\begin{definition}
For any real $s$, let $\fonctionsansdef{\phi (s,\cdot)}{\R ^d}{\R ^d}$ be a diffeomorphism. Then, $\Phi = \{ \phi (s,\cdot) \}_{s \in \R}$ is a one parameter group of diffeomorphisms if it satisfies 
\begin{enumerate}
\item $\phi (0,\cdot) = Id_{\R ^d}$, 
\item $\forall s$, $s' \in \R, \; \phi (s,\cdot) \circ \phi (s',\cdot) = \phi (s+s',\cdot) $,
\item $\phi$ is of class $\CC^2$.
\end{enumerate}
\end{definition}

Classical examples of one parameter groups of diffeomorphisms are given by translations and rotations.\\

The action of a one parameter group of diffeomorphisms on a Lagrangian allows to define the notion of a symmetry for a fractional Lagrangian system:

\begin{definition}
\label{definvariancefrac}
Let $\Phi = \{ \phi (s,\cdot) \}_{s \in \R}$ be a one parameter group of diffeomorphisms and let $L$ be a Lagrangian. $L$ is said to be $D^\alpha_-$-invariant under the action of $\Phi$ if it satisfies:
\begin{equation}
\forall q \text{ solution of \eqref{elf}},  \; \forall s \in \R, \; L \Big( \phi(s,q), D^\alpha_- \big( \phi (s,q) \big),t \Big) = L \big( q,D^\alpha_- q,t \big) .
\end{equation}
\end{definition}

% Theorem \ref{thmcresson} et \eqref{eqthmcress} et \ref{thmcresson}

\subsection{A fractional Noether's theorem}\label{section13}

Cresson \cite{cres6} and Torres \textit{et al.} \cite{torr3} proved the following result:

\begin{lemma}
\label{lemcresson}
Let $L$ be a Lagrangian $D^\alpha_-$-invariant under the action of a one parameter group of diffeomorphisms $\Phi = \{ \phi (s,\cdot) \}_{s \in \R}$. Then, the following equality holds for any solution $q$ of \eqref{elf}:
\begin{equation}\label{eqlemcresson}
D^\alpha_- \left( \dfrac{\partial \phi}{\partial s} (0,q) \right) \cdot \dfrac{\partial L}{\partial v}(q,D^\alpha_- q,t) - \dfrac{\partial \phi}{\partial s} (0,q) \cdot D^\alpha_+ \left( \dfrac{\partial L}{\partial v}(q,D^\alpha_- q,t) \right)  = 0 .
\end{equation}
\end{lemma}

In the classical case $\alpha =1$, the classical Leibniz formula allows to rewrite \eqref{eqlemcresson} as the derivative of a product. Precisely, for $\alpha =1$, Lemma \ref{lemcresson} leads to the classical Noether's theorem given by:

\begin{theorem}
[Classical Noether's theorem]
\label{thmnoetherclass}
Let $L$ be a Lagrangian $d/dt$-invariant under the action of a one parameter group of diffeomorphisms $\Phi = \{ \phi (s,\cdot) \}_{s \in \R}$. Then, the following equality holds for any solution $q$ of {\rm (EL${}^1$)}:
\begin{equation}\label{eqthmnoetherclass}
\dfrac{d}{dt} \left( \dfrac{\partial \phi}{\partial s} (0,q)  \cdot \dfrac{\partial L}{\partial v}(q,\dot{q},t) \right)  = 0, 
\end{equation}
where $\dot{q}$ is the classical derivative of $q$, \textit{i.e.} $dq/dt$.
\end{theorem}

Theorem \ref{thmnoetherclass} provides an explicit constant of motion for any classical Lagrangian systems admitting a symmetry. In the fractional case, such a simple formula allowing to rewrite \eqref{eqlemcresson} as a total derivative with respect to $t$ is not known yet. Nevertheless, the next theorem provides such a formula in an explicit form:

\begin{theorem}[Transfer formula]
\label{thmleibizfrac}
Let $f$, $g \in \CC^{\infty} ([a,b],\R^d)$ assuming the following Condition (C):
\begin{center}
the sequences of functions $ (I^{p-\alpha}_- f \cdot g^{(p)})_{p \in \N^*} $ and \\ $ (f^{(p)} \cdot I^{p-\alpha}_+ g )_{p \in \N^*} $ converge uniformly to $0$ on $ [a,b] $.
\end{center}
Then, the following equality holds:
\begin{equation}\label{eqthmleibnizfrac}
D^\alpha_- f \cdot g - f \cdot D^\alpha_+ g = \dfrac{d}{dt} \left[ \di \sum_{r=0}^{\infty} (-1)^r I^{r+1-\alpha}_- f \cdot g^{(r)} + f^{(r)} \cdot I^{r+1-\alpha}_+ g \right]. 
\end{equation}
\end{theorem}

\begin{proof} 
Let $f$, $g \in \CC^{\infty} ([a,b],\R^d)$. By induction and using the classical Leibniz formula at each step, we prove that for any $p \in \N^*$, the following equalities both hold:
\begin{equation}\label{ddd}
D^\alpha_- f \cdot g = (-1)^p I^{p-\alpha}_- f \cdot g^{(p)} + \dfrac{d}{dt} \left[ \di \sum_{r=0}^{p-1} (-1)^r I^{r+1-\alpha}_- f \cdot g^{(r)} \right]
\end{equation}
and
\begin{equation}
- f \cdot D^\alpha_+ g = f^{(p)} \cdot I^{p-\alpha}_+ g +  \dfrac{d}{dt} \left[ \di \sum_{r=0}^{p-1} f^{(r)} \cdot I^{r+1-\alpha}_+ g \right].
\end{equation}
For any $p \in \N^*$, we define $u_p$ the following function:
\begin{equation}
u_p = \di \sum_{r=0}^{p-1} (-1)^r I^{r+1-\alpha}_- f \cdot g^{(r)} . 
\end{equation}
According to \eqref{ddd}, for any $p \in \N^*$, $ \dot{u}_p = D^\alpha_- f \cdot g - (-1)^p I^{p-\alpha}_- f \cdot g^{(p)} $. Hence, the assumption made on the sequence of functions $ (I^{p-\alpha}_- f \cdot g^{(p)})_{p \in \N^*} $ implies that the sequence $(\dot{u}_p)_{p \in \N^*}$ converges uniformly to $ D^\alpha_- f \cdot g $ on $ [a,b]$. Moreover, let us note that $(u_p)_{p \in \N^*}$ point-wise converges in $t=a$. Finally, we can conclude that the sequence $(u_p)_{p \in \N^*}$ converges uniformly on $ [a,b]$ to a function $u$ equal to
\begin{equation}
u = \di \sum_{r=0}^{\infty} (-1)^r I^{r+1-\alpha}_- f \cdot g^{(r)}  
\end{equation}
and satisfying $ \dot{u} = D^\alpha_- f \cdot g $. Similarly, one can prove that:
\begin{equation}
\dfrac{d}{dt} \left[ \di \sum_{r=0}^{\infty} f^{(r)} \cdot I^{r+1-\alpha}_+ g \right] = - f \cdot D^\alpha_+ g ,
\end{equation}
which completes the proof.
\end{proof}
Thus, combining Lemma \ref{lemcresson} and this transfer formula, we prove:

\begin{theorem}[A fractional Noether's theorem]
\label{thmnoetherfrac}
Let $L$ be a Lagrangian $D^\alpha_-$-invariant under the action of a one parameter group of diffeomorphisms $\Phi = \{ \phi (s,\cdot) \}_{s \in \R}$. Let $q$ be a solution of \eqref{elf} and let $f$ and $g$ denote:
\begin{equation}
f = \dfrac{\partial \phi}{\partial s} (0,q) \quad \text{and} \quad g = \dfrac{\partial L}{\partial v}(q,D^\alpha_- q,t).
\end{equation}
If $f$ and $g$ satisfy Condition (C), then the following equality holds:
\begin{equation}
\label{eqthmnoetherfrac}
\dfrac{d}{dt} \left[ \di \sum_{r=0}^{\infty} (-1)^r I^{r+1-\alpha}_- f \cdot g^{(r)} + f^{(r)} \cdot I^{r+1-\alpha}_+ g \right] = 0 .
\end{equation}
\end{theorem}

This theorem provides an explicit algorithmic way to compute a constant of motion for any fractional Lagrangian systems admitting a symmetry. An arbitrary closed approximation of this quantity can be obtained with a truncature of the infinite sum.

\subsection{Comments}

\subsubsection{About previous studies}\label{section141}
Previous results in this direction have been obtained by several authors, see \cite{torr3,atan,musl}. Let us note that these authors consider symmetries modifying also the time variable in contrary to this paper. Nevertheless, considering symmetries not involving a time transformation, Lemma \ref{lemcresson} and their results coincide. Then, let us make the following comments explaining the interest of Theorem \ref{thmnoetherfrac} in this case:
\begin{itemize}
\item In \cite{torr3}, Torres \textit{et al.} define a \textit{fractional conservation law} for fractional Lagrangian systems. Precisely, denoting by $D^{\alpha}_0$ the bilinear operator given by:
\begin{equation}
D^\alpha_0 (f,g) = D^\alpha_- f \cdot g - f \cdot D^\alpha_+ g,
\end{equation} 
authors give the following definition: a function $\fonction{C}{\R^d \times \R^d \times [a,b]}{\R}{(x,v,t)}{C(x,v,t)}$ is a \textit{fractional-conserved quantity} of a fractional system if it is possible to write $C$ in the form $C(x,v,t)=C^1 (x,v,t) \cdot C^2 (x,v,t) $ with $D^\alpha_0 \big( C^1 (q,D^\alpha_- q,t), C^2 (q,D^\alpha_- q,t) \big) = 0 $ for any solution $q$ of the considered system. They deduce easily that $\partial \phi / \partial s (0,x) \cdot \partial L/\partial v (x,v,t)$ is then a fractional-conserved quantity of \eqref{elf}. Two important difficulties appear:
\begin{itemize}
\item Although all these notions and results coincide with the classical ones when $\alpha =1$, we do not have that $D^\alpha_0 (f,g) = 0$ implies $f \cdot g$ is constant. 
\item The decomposition $C=C^1 \cdot C^2$ is not unique and it leads to many issues. Indeed, the scalar product is commutative and then $C=C^1 \cdot C^2 = C^2 \cdot C^1$. Nevertheless, the operator $D^\alpha_0$ is not symmetric! Hence, in the result of Torres \textit{et al}, we have:
\begin{equation}
C(x,v,t) = \dfrac{\partial \phi}{\partial s} (0,x) \cdot \dfrac{\partial L}{\partial v} (x,v,t) = \dfrac{\partial L}{\partial v} (x,v,t) \cdot \dfrac{\partial \phi}{\partial s} (0,x)
\end{equation}
is a fractional-conserved quantity of \eqref{elf} but we do not have:
\begin{equation}
D^\alpha_0 \left(\dfrac{\partial L}{\partial v} (q,D^\alpha_- q,t), \dfrac{\partial \phi}{\partial s} (0,q) \right) = 0.
\end{equation}
\end{itemize}
Let us note that the work developed in \cite{musl} is similar to this previous one.
\item In \cite{atan}, Atanackovi\`c \textit{et al.} obtain in \cite[Theorem 15]{atan} a formulation of a constant of motion for a fractional Lagrangian systems admitting a symmetry: for any solution $q$ of \eqref{elf}, let $f = \partial \phi / \partial s (0,q)$ and $g = \partial L / \partial v (q,D^\alpha_- q,t)$, the following element $ t \longmapsto \int_a^t D^\alpha_- f \cdot g - f \cdot D^\alpha_+ g \; dy $ is constant on $[a,b]$. This result is then unsatisfactory because the constant of motion is not explicit. 
\end{itemize}
Hence, these definitions and results define \textit{implicitly} a constant of motion by functional relations (see \cite{torr3,musl}) or integral relations (see \cite{atan}). Theorem \ref{thmnoetherfrac} is then interesting because it derives \textit{explicitly} a constant of motion in terms of symmetry group and Lagrangian. \\

Let us make the following remark concerning the time transformation. In each of these papers \cite{torr3,atan,musl}, authors finally prove that a symmetry of a fractional Lagrangian system (with time transformation) implies an equality of the following type:
\begin{equation}
D^\alpha_- f \cdot g - f \cdot D^\alpha_+ g = 0,
\end{equation}
where $f$ depends on the symmetry group and $g$ depends on the Lagrangian $L$. Consequently, the transfer formula given in Theorem \ref{thmleibizfrac} is also relevant and then an explicit constant of motion can be obtained. Finally, the study developed in this paper is also available in the case of time transformation and it will be done in a forthcoming paper. \\

\subsubsection{About Condition (C)} 
Condition (C) could seem strong or too particular. In order to make it more concrete and understand what classes of functions satisfy it, we give the two following sufficient conditions:

\begin{proposition}\label{propcondc}
Let $f$, $g \in \CC^{\infty} ([a,b],\R^d)$. 
\begin{enumerate}
\item If $f$ and $g$ satisfy the two following conditions:
\begin{equation}
\max\limits_{t \in [a,b]} \left( \dfrac{(b-t)^{p-1}}{(p-1)!} \Vert f^{(p)} (t) \Vert \right) \limit\limits_{p \to \infty} 0 \quad \text{and} \quad \max\limits_{t \in [a,b]} \left( \dfrac{(t-a)^{p-1}}{(p-1)!} \Vert g^{(p)} (t) \Vert \right) \limit\limits_{p \to \infty} 0
\end{equation}
then $f$ and $g$ satisfy Condition (C).
\item If there exists $M > 0$ such that:
\begin{equation}
\forall p \in \N^*, \; \forall t \in [a,b], \; \Vert f^{(p)} (t) \Vert \leq M \quad \text{and} \quad  \Vert g^{(p)} (t) \Vert \leq M
\end{equation}
then $f$ and $g$ satisfy Condition (C).
\end{enumerate}
\end{proposition}

\begin{proof}
1. Let us prove that the sequence of functions $(I^{p-\alpha}_- f \cdot g^{(p)})_{p \in \N^*}$ converges uniformly to $0$. The proof is similar for the sequence $( f^{(p)} \cdot I^{p-\alpha}_- g)_{p \in \N^*}$. Let us denote $M = \max\limits_{t \in [a,b]} (\Vert f(t) \Vert)$. For any $ p \in \N^*$ and any $t \in [a,b]$:
\begin{equation}
\begin{array}{rcl}
\left| I^{p-\alpha}_- f (t) \cdot g^{(p)}(t) \right| & = & \left| \dfrac{g^{(p)}(t)}{\Gamma (p-\alpha )} \cdot \di \int_a^t (t-y)^{p-\alpha-1} f(y) \; dy \right| \\
& \leq & M \dfrac{(t-a)^{p-\alpha}}{\Gamma (p+1-\alpha )} \Vert g^{(p)} (t) \Vert.
\end{array}
\end{equation}
Finally, since $\Gamma (p+1-\alpha ) \geq (p-1)! \, \Gamma (2-\alpha )$:
\begin{equation}
\begin{array}{rcl}
\left| I^{p-\alpha}_- f (t) \cdot g^{(p)}(t) \right| & \leq & M \dfrac{(t-a)^{1-\alpha}}{\Gamma (2-\alpha )} \dfrac{(t-a)^{p-1}}{(p-1)!} \Vert g^{(p)} (t) \Vert \\
& \leq & M \dfrac{(b-a)^{1-\alpha}}{\Gamma (2-\alpha )} \dfrac{(t-a)^{p-1}}{(p-1)!} \Vert g^{(p)} (t) \Vert,
\end{array}
\end{equation}
which concludes the proof. \\
2. It follows from the first result.
\end{proof}

For example, if $f$ is polynomial and $g$ is the exponential function, one can prove that $f$ and $g$ satisfy the point $2$ of Proposition \ref{propcondc} and then Condition (C). If $g(t)=1/t$ on $[a,b]$ with $a >0$, one can notice that the point $2$ of Proposition \ref{propcondc} is not satisfied anymore. Nevertheless, the point $1$ is true and consequently the functions $f$ and $g$ satisfy Condition (C). In general, every couple of analytic functions satisfy Condition (C) and consequently \eqref{eqthmleibnizfrac} is true.

\section{A discrete fractional Noether's theorem}
\label{section2} 

We study the existence of discrete conservation laws for discrete fractional Lagrangian systems in the sense of \cite{bour}. Using the same strategy as in the continuous case, we introduce the notion of discrete symmetry and prove a discrete fractional Noether's theorem providing an {\it explicit computable} discrete constant of motion. 

\subsection{Reminder about discrete fractional Lagrangian systems} 

We follow the definition of discrete fractional Lagrangian systems given in \cite{bour} to which we refer for more details. \\

Let us consider $N \in \N^*$, $h = (b-a)/N$ the step size of the discretization and $\tau = (t_k)_{k=0,\dots,N} = (a+kh)_{k=0,\dots,N} $ the usual partition of the interval $[a,b]$. Let us define $\Delta_-^{\alpha}$ and $\Delta_+^{\alpha}$ the following discrete analogous of $D^{\alpha}_{-}$ and $D^{\alpha}_{+}$ respectively:
\begin{equation}
\fonction{\Delta^{\alpha} _-}{( \R^d ) ^{N+1}}{( \R^d ) ^{N}}{Q}{\left( \dfrac{1}{h^{\alpha}} \displaystyle \sum_{r=0}^{k} \alpha_r Q_{k-r} \right) _{k=1,..,N},}
\end{equation}
\begin{equation}
\fonction{\Delta^{\alpha} _+}{( \R^d ) ^{N+1}}{( \R^d ) ^{N}}{Q}{\left( \dfrac{1}{h^{\alpha}} \displaystyle \sum_{r=0}^{N-k} \alpha_r Q_{k+r} \right) _{k=0,..,N-1},}
\end{equation}
where the elements $(\alpha_r)_{r \in \N}$ are defined by $\alpha_0 = 1$ and 
\begin{equation}\label{eqalphar}
\forall r \in \N^*, \; \alpha_r = \dfrac{(-\alpha)(1-\alpha)\dots(r-1-\alpha)}{r!}.
\end{equation}
These discrete fractional operators are approximations of the continuous ones. We refer to \cite{podl} for more details. \\

A discrete fractional Euler-Lagrange equation, of unknown $Q \in ( \R^d ) ^{N+1}$, is defined by:
\begin{equation}
\label{elfh}\tag{EL${}^\alpha_h$}
 \di \frac{\partial L}{\partial x} ( Q,\Delta^{\alpha} _- Q,\tau ) + \Delta^{\alpha} _+ \left( \frac{\partial L}{\partial v} ( Q,\Delta^{\alpha} _- Q,\tau) \right) = 0, 
\end{equation}
where $L$ is a Lagrangian. In such a case, we speak of a discrete Lagrangian system.\\

The terminology is justified by the fact that solutions of \eqref{elfh} correspond to discrete critical points of the discrete fractional Lagrangian functional defined by:
\begin{equation}
\label{eqdlagfuncfrac}
\fonction{\mathcal{L}^{\alpha}_h}{( \R^{d} )^{N+1}}{\R}{Q}{h \di \sum_{k=1}^N L \big( Q_k,(\Delta^{\alpha} _{-} Q)_k,t_k \big).}
\end{equation}

We refer to \cite{bour} for a proof.\\

In the case $\alpha = 1$, the discrete operators $\Delta^1_\pm$ correspond to the implicit and explicit Euler approximations of $d/dt$ and \eqref{elfh} is just the discrete Euler-Lagrange equation obtained in \cite{mars,lubi}.

\subsection{Discrete symmetries} 

Here again, a discrete symmetry of a discrete fractional Lagrangian system is based on the action of a one parameter group of transformations on the associated Lagrangian:

\begin{definition}
\label{definvariancedfrac}
Let $\Phi = \{ \phi (s,\cdot) \}_{s \in \R}$ be a one parameter group of diffeomorphisms and let $L$ be a Lagrangian. $L$ is said to be $\Delta^\alpha_-$-invariant under the action of $\Phi$ if it satisfies:
\begin{equation}\label{eqinvarfracd}
\forall Q \text{ solution of \eqref{elfh}},  \; \forall s \in \R, \; L \Big( \phi(s,Q), \Delta^\alpha_- \big( \phi (s,Q) \big),\tau \Big) = L \big( Q,\Delta^\alpha_- Q,\tau \big) .
\end{equation}
\end{definition}

Then, we can prove the following discrete version of Lemma \ref{lemcresson}:

\begin{lemma}
\label{lemcressond}
Let $L$ be a Lagrangian $\Delta^\alpha_-$-invariant under the action of a one parameter group of diffeomorphisms $\Phi = \{ \phi (s,\cdot) \}_{s \in \R}$. Then, the following equality holds for any solution $Q$ of \eqref{elfh}:
\begin{equation}\label{eqlemcressond}
\Delta^\alpha_- \left( \dfrac{\partial \phi}{\partial s} (0,Q) \right) \cdot \dfrac{\partial L}{\partial v}(Q,\Delta^\alpha_- Q,\tau) - \dfrac{\partial \phi}{\partial s} (0,Q) \cdot \Delta^\alpha_+ \left( \dfrac{\partial L}{\partial v}(Q,\Delta^\alpha_- Q, \tau) \right)  = 0 .
\end{equation}
\end{lemma}

\begin{proof} 
This proof is a direct adaptation to the discrete case of the proof of Lemma \ref{lemcresson}. Let $Q \in (\R^d)^{N+1}$ be a solution of \eqref{elfh}. Let us differentiate equation \eqref{eqinvarfracd} with respect to $s$ and invert the operators $\Delta^{\alpha}_-$ and $\partial / \partial s$. We finally obtain for any $s \in \R$ and any $k \in \{ 1,\dots,N-1\}$:
\begin{multline}\label{astast}
\Delta^{\alpha}_- \left( \dfrac{\partial \phi}{\partial s} (s,Q) \right)_k \cdot \dfrac{\partial L}{\partial v} \Big( \phi (s,Q_k), \Delta^\alpha_- \big( \phi (s,Q_k) \big),t_k \Big) \\ + \dfrac{\partial L}{\partial x} \Big( \phi (s,Q_k), \Delta^\alpha_- \big( \phi (s,Q_k) \big),t_k \Big) \cdot \dfrac{\partial \phi}{\partial s} (s,Q_k) = 0  .
\end{multline}
Since $\phi (0,\cdot) = Id_{\R^d}$ and $Q$ is a solution of \eqref{elfh}, taking $s=0$ in \eqref{astast} leads to \eqref{eqlemcressond}.
\end{proof}

\subsection{A discrete fractional Noether's theorem} 

Let us remind that our aim is to express explicitly a discrete constant of motion for discrete fractional Lagrangian systems admitting a discrete symmetry. As in the continuous case, our result is based on Lemma \ref{lemcressond}. Let us note that the following implication holds:
\begin{equation}\label{dconstant}
\forall F \in \R^{N+1}, \; \Delta^1_- F = 0 \Longrightarrow \exists c \in \R, \; \forall k=0,\dots,N, \; F_k = c.
\end{equation}
Namely, if the discrete derivative of $F$ vanishes, then $F$ is constant. Our aim is then to write \eqref{eqlemcressond} as a discrete derivative (\textit{i.e.} as $\Delta^1_-$ of an explicit quantity). \\

We first introduce some notations and definitions. The shift operator denoted by $\sigma$ is defined by
\begin{equation}
\fonction{\sigma}{(\R^d)^{N+1}}{(\R^d)^{N+1}}{Q}{\sigma(Q) = (Q_{k+1})_{k=0,\dots,N}}
\end{equation}
with the convention $Q_{N+1}=0$. We also introduce the following square matrices of length $(N+1)$. First, $A_{1} = - Id_{N+1}$ and then, for any $r \in \{ 2,\dots,N-1 \}$, the square matrices $A_r \in \mathcal{M}_{N+1}$ defined by:
\begin{equation}
\forall i,j=0,\dots,N, \; (A_r)_{i,j} = 
\left\lbrace \begin{array}{lcl}
0 & \text{if} & i=0 \\
\delta_{\{ j=0 \}} \delta_{\{  r \leq i \}} - \delta_{\{ 0 \leq i-j \leq r-1\}} \delta_{\{  1 \leq j \leq N-r \}} & \text{if} & 1 \leq i \leq N-1 \\
(A_r)_{N-1,j} & \text{if} & i = N
\end{array} \right.
\end{equation}
where $\delta$ is the Kronecker symbol. \\

For example, for $N=5$, the matrices $A_r$ are given by: $A_1 = -Id_6$ and
\begin{scriptsize}
\begin{equation*}
A_2 = \left(
\begin{array}{cccccc}
0 & 0 & 0 & 0 & 0 & 0 \\
0 & -1& 0 & 0 & 0 & 0 \\
1 & -1& -1& 0 & 0 & 0 \\
1 & 0 & -1& -1& 0 & 0 \\
1 & 0 & 0 & -1& 0 & 0 \\
1 & 0 & 0 & -1& 0 & 0 
\end{array} \right) , 
A_3 = \left(
\begin{array}{cccccc}
0 & 0 & 0 & 0 & 0 & 0 \\
0 & -1& 0 & 0 & 0 & 0 \\
0 & -1& -1& 0 & 0 & 0 \\
1 & -1& -1& 0 & 0 & 0 \\
1 & 0 & -1& 0 & 0 & 0 \\
1 & 0 & -1& 0 & 0 & 0 
\end{array} \right) , 
A_4 = \left(
\begin{array}{cccccc}
0 & 0 & 0 & 0 & 0 & 0 \\
0 & -1& 0 & 0 & 0 & 0 \\
0 & -1& 0 & 0 & 0 & 0 \\
0 & -1& 0 & 0 & 0 & 0 \\
1 & -1& 0 & 0 & 0 & 0 \\
1 & -1& 0 & 0 & 0 & 0 
\end{array} \right) 
.
\end{equation*}
\end{scriptsize}

Considering these previous elements, we state the following result:

\begin{theorem}
[A discrete fractional Noether's theorem]
\label{thmnoetherdfrac}
Let $L$ be a Lagrangian $\Delta^{\alpha}_-$-invariant under the action of a one parameter group of diffeomorphisms $\Phi = \{ \phi (s,\cdot) \}_{s \in \R}$. Then, the following equality holds for any solution $Q$ of \eqref{elfh}:
\begin{equation}\label{eqthmdfnoehter}
\Delta^1_- \left[ \di \sum_{r=1}^{N-1} \alpha_r A_r \left( \dfrac{\partial \phi}{\partial s} (0,Q) \cdot \sigma^{r}  \Big( \dfrac{\partial L}{\partial v}(Q,\Delta^\alpha_- Q,\tau) \Big) \right) \right] = 0.
\end{equation}
\end{theorem} 

Combining \eqref{dconstant} and \eqref{eqthmdfnoehter}, Theorem \ref{thmnoetherdfrac} provides a constant of motion for discrete fractional Lagrangian systems admitting a symmetry. Let us note that the discrete conservation law is not only explicit but computable in finite time. We give an example in the next section.\\

Before giving its proof, we remark that Theorem \ref{thmnoetherdfrac} takes a particular simple expression when $\alpha =1$. Indeed, since $\alpha_r = 0$ for any $r \geq 2$ in this case, we obtain:

\begin{theorem}[Discrete classical Noether's theorem]\label{thmnoetherdclass}
Let $L$ be a Lagrangian $\Delta^1_-$-invariant under the action of a one parameter group of diffeomorphisms $\Phi = \{ \phi (s,\cdot) \}_{s \in \R}$. Then, the following equality holds for any solution $Q$ of {\rm (EL${}^1_h$)}:
\begin{equation}\label{dcnoether}
\Delta^1_- \left[ \dfrac{\partial \phi}{\partial s} (0,Q) \cdot \sigma \left( \dfrac{\partial L}{\partial v} (Q,\Delta^1_- Q,\tau ) \right) \right]= 0 .
\end{equation}
\end{theorem}

This result is a reformulation of the discrete Noether's theorem proved in \cite{mars,lubi}. It corresponds to our Lemma \ref{lemcressond} with $\alpha=1$ using the following discrete Leibniz formula:
\begin{equation}\label{eqleibnizd}
\forall F,G \in (\R^d)^{N+1}, \; \Delta^1_- \big( F \cdot \sigma (G) \big) = \Delta^1_- F \cdot G - F \cdot \Delta^1_+ G . 
\end{equation}
Now, let us prove Theorem \ref{thmnoetherdfrac}:
\begin{proof}[Proof of Theorem \ref{thmnoetherdfrac}]
According to Lemma \ref{lemcressond}, equation \eqref{eqlemcressond} holds for any $k=1,\dots,N-1$. Let $F = \partial \phi / \partial s (0,Q)$ and $G = \partial L / \partial v (Q,\Delta^\alpha_- Q,\tau)$. Let us multiply \eqref{eqlemcressond} by $h^\alpha$ and obtain the following equality for any $k=1,\dots,N-1$:
\begin{equation}
\left( \di \sum_{r=0}^k \alpha_r F_{k-r} \right) \cdot G_k - F_k \cdot \left( \di \sum_{r=0}^{N-k} \alpha_r G_{k+r} \right) = 0,
\end{equation}
and since $\alpha_0 = 1$:
\begin{equation}
\alpha_1 ( F_{k-1} \cdot G_k - F_k \cdot G_{k+1} )_k + \left( \di \sum_{r=2}^k \alpha_r F_{k-r} \right) \cdot G_k - F_k \cdot \left( \di \sum_{r=2}^{N-k} \alpha_r G_{k+r} \right) = 0 ,
\end{equation}
then:
\begin{equation}\label{jk}
\alpha_1 \Delta^1_- \big( F \cdot \sigma(G) \big)_k = \dfrac{1}{h} \left[ \left( \di \sum_{r=2}^k \alpha_r F_{k-r} \right) \cdot G_k - F_k \cdot \left( \di \sum_{r=2}^{N-k} \alpha_r G_{k+r} \right) \right].
\end{equation}
Let us denote $J_k$ the right term of \eqref{jk} for any $k=1,\dots,N-1$. We are going to write $J_k$ as the discrete derivative of its discrete anti-derivative: it corresponds to the discrete version of the method of Atanackovi\`c in \cite{atan}. For any $k=1,\dots,N-1$, we obtain $ J_k = (\Delta^1_- H)_k $ where $H_i := h \sum_{j=1}^{i} J_j$ for any $i=0,\dots,N$ with the convention $ H_0 = J_N = 0 $. Let us provide an explicit formulation of the element $H = (H_i)_{i=0,\dots,N}$. 

$ \bullet $ \textit{Case} $1 \leq i \leq N-1$:
\begin{equation}\label{test2}
H_i = \di \sum_{j=1}^i \left[ \di \sum_{r=2}^j \alpha_r F_{j-r} \cdot G_j - \di \sum_{r=2}^{N-j} \alpha_r F_{j} \cdot G_{j+r} \right] = \di \sum_{j=2}^i \sum_{r=2}^j \alpha_r F_{j-r} \cdot G_j - \di \sum_{j=1}^i \sum_{r=2}^{N-j} \alpha_r F_{j} \cdot G_{j+r}.
\end{equation}
As $ \sum_{j=2}^i \sum_{r=2}^j = \sum_{r=2}^i \sum_{j=r}^{i}$, we have:
\begin{equation}
\di \sum_{j=2}^i \sum_{r=2}^j \alpha_r F_{j-r} \cdot G_j = \sum_{r=2}^i \sum_{j=r}^{i} \alpha_r F_{j-r} \cdot G_j = \sum_{r=2}^i \sum_{j=0}^{i-r} \alpha_r F_{j} \cdot G_{j+r} = \di \sum_{r=2}^i \sum_{j=0}^{i-r} \alpha_r \big( F \cdot \sigma^r (G) \big)_j.
\end{equation}
Then, since $ \sum_{j=1}^i \sum_{r=2}^{N-j} = \sum_{j=1}^i \sum_{r=2}^{N-i} + \sum_{j=1}^i \sum_{r=N+1-i}^{N-j} = \sum_{r=2}^{N-i} \sum_{j=1}^{i} + \sum_{r=N+1-i}^{N-1} \sum_{j=1}^{N-r} $:
\begin{equation}\label{test}
H_i = \di \sum_{r=2}^i \sum_{j=0}^{i-r} \alpha_r \big( F \cdot \sigma^r (G) \big)_j - \sum_{r=2}^{N-i} \sum_{j=1}^{i} \alpha_r \big( F \cdot \sigma^r (G) \big)_j  - \sum_{r=N+1-i}^{N-1} \sum_{j=1}^{N-r} \alpha_r \big( F \cdot \sigma^r (G) \big)_j.
\end{equation}
Finally, we have:
\begin{equation}
H_i = \di \sum_{r=2}^{N-1} \sum_{j=0}^N \alpha_r A_r (i,j) \big( F \cdot \sigma^r (G) \big)_j ,
\end{equation}
where the elements $\big( A_r (i,j) \big)$ are defined for $r=2,\dots,N-1$ and $j=0,\dots,N$ as the real coefficients in front of $\alpha_r \big( F \cdot \sigma^r (G) \big)_j$. Our aim is then to express the values of these elements. From \eqref{test}, we have for any $r=2,\dots,N-1$ and any $j=0,\dots,N$:
\begin{equation}
A_r(i,j) = \delta_{\{r \leq i\}} \delta_{\{ 0 \leq j \leq i-r \}} - \delta_{\{ r \leq N-i \}} \delta_{\{ 1 \leq j \leq i \}} - \delta_{\{  N+1 - i \leq r \}} \delta_{\{ 1 \leq j \leq N-r \}}.
\end{equation}
For example, for $j=0$, we have:
\begin{equation}\label{j0}
\forall r=2,\dots,N-1, \; A_r(i,0) = \delta_{\{ r \leq i \} }.
\end{equation}
Let $r \in \{2,\dots,N-1\}$ and $j \in \{1,\dots,N\}$. Let us prove that:
\begin{equation}\label{goodj1}
A_r(i,j) = -\delta_{\{ 0 \leq i-j \leq r-1\}} \delta_{\{  j \leq N-r \}} = -\delta_{\{j \leq i\}} \delta_{\{i-j \leq r-1\}} \delta_{\{  j \leq N-r \}}.
\end{equation}
Let us see the four following cases:
\begin{itemize}
\item if $j>i$, then $j>i-r$. Moreover, if $N+1-i \leq r$ then $j > N-r$. In this case, $A_r(i,j) = 0-0-0=0$.
\item if $i-j > r -1$ then $r \leq i$, $j \leq i-r$, $j\leq i$ and $j \leq N-r$. Then, $A_r (i,j)=1-1-0=0$ or $A_r (i,j)=1-0-1=0$ depending on $r \leq N-i$ or $N+1-i \leq r$. Finally, in this case, $A_r (i,j)=0$.
\item if $j > N-r$ then $j > i-r$. Moreover, if $r \leq N-i$ then $j > i$. In this case, $A_r(i,j) = 0-0-0=0$.
\item if $ j \leq i $, $ i-j \leq r-1 $ and $ j \leq N-r $, then $i-r <j$. In this case, $A_r (i,j)=0-1-0=-1$ or $A_r (i,j)=0-0-1=-1$ depending on $r \leq N-i$ or $N+1-i \leq r$. Finally, in this case, $A_r (i,j)=-1$.
\end{itemize}
Consequently, \eqref{goodj1} holds for any $r \in \{2,\dots,N-1\}$ and any $j \in \{1,\dots,N\}$. Finally, from \eqref{goodj1} and \eqref{j0}, we have:
\begin{equation}\label{goodj0}
\forall r=2,\dots,N-1, \; \forall j=0,\dots,N, \; A_r(i,j) = \delta_{\{ j=0 \}} \delta_{\{  r \leq i \}} - \delta_{\{ 0 \leq i-j \leq r-1\}} \delta_{\{  1 \leq j \leq N-r \}} .
\end{equation}
$ \bullet $ \textit{Case} $i=0$ or $i=N$. As $H_0 = 0$, for any $r=2,\dots,N-1$ and any $j=0,\dots,N$, we define $A_r(0,j)=0$. As $H_N = H_{N-1}$, for any $r=2,\dots,N-1$ and any $j=0,\dots,N$, we define $A_r(N,j)=A_r(N-1,j)$. \\

Hence, for any $r=2,\dots,N-1$, the elements $\big(A_r(i,j)\big)$ are defined for $i=0,\dots,N$ and $j=0,\dots,N$. Then, we denote by $A_r$ the matrix $\big(A_r(i,j)\big)_{0\leq i,j \leq N} \in \mathcal{M}_{N+1}$. Finally, from \eqref{jk}, we have proved that $ \Delta^1_- \big( - \alpha_1 F\cdot\sigma(G) + H \big) = 0 $ where $H = \sum_{r=2}^{N-1} \alpha_r A_r \big( F \cdot \sigma^r (G) \big)$. Finally, denoting $A_1 = -Id_{N+1} \in \mathcal{M}_{N+1}$, we conclude the proof.
\end{proof}

\begin{remark}
Using the discrete Leibniz formula given by \eqref{eqleibnizd}, another choice in order to give a discrete version of \eqref{eqthmnoetherfrac} from Lemma \ref{lemcressond} would be to apply the discrete version of the method used in Section \ref{section13}. Nevertheless, we would encounter many numerical difficulties. Firstly, such a method would imply the use of the operator $\Delta^p_-$ but this operator approximates the operator $(d/dt)^p$ only for $t_k$ with $k \geq p$. For the $p$ first terms, we obtain in general a numerical blow up. Secondly, the use of operator $\Delta^p_-$ implies the use of $h^{-p}$. Hence, for a large enough $p$, we exceed the machine precision. 
\end{remark}

\subsection{Example: the discrete fractional harmonic oscillator} 

We consider the classical bi-dimensional example ($d=2$) of the quadratic Lagrangian $L (x,v,t) = (x^2 + v^2)/2$ with $[a,b]=[0,1]$ and $\alpha = 1/2$. Then, $L$ is $\Delta^\alpha_-$-invariant under the action of the rotations given by:
\begin{equation}\label{rotation}
\fonction{\phi}{\R \times \R^2}{\R^2}{(s,x_1,x_2)}{\left( \begin{array}{cc} \cos (s) & - \sin (s ) \\ \sin (s) & \cos (s )  \end{array} \right) \left( \begin{array}{c} x_1 \\ x_2 \end{array} \right).}
\end{equation}
We choose $N=600$, $Q_0 = (1,2)$ and $Q_N = (2,1)$. Let $F$ and $G$ denote:
\begin{equation}
F = \dfrac{\partial \phi}{\partial s} (0,Q) = (-Q^2,Q^1) \quad \text{and} \quad G = \dfrac{\partial L}{\partial v}(Q,\Delta^\alpha_- Q,\tau) = (\Delta^\alpha_- Q^1,\Delta^\alpha_- Q^2).
\end{equation}
The computation of \eqref{elfh} gives the following graphics: 
\begin{figure}[h]
  \centering
  \subfloat[Solution $Q= (Q^1,Q^2)$ of \eqref{elfh}]{\label{Q1Q2} \includegraphics[width=0.5\textwidth]{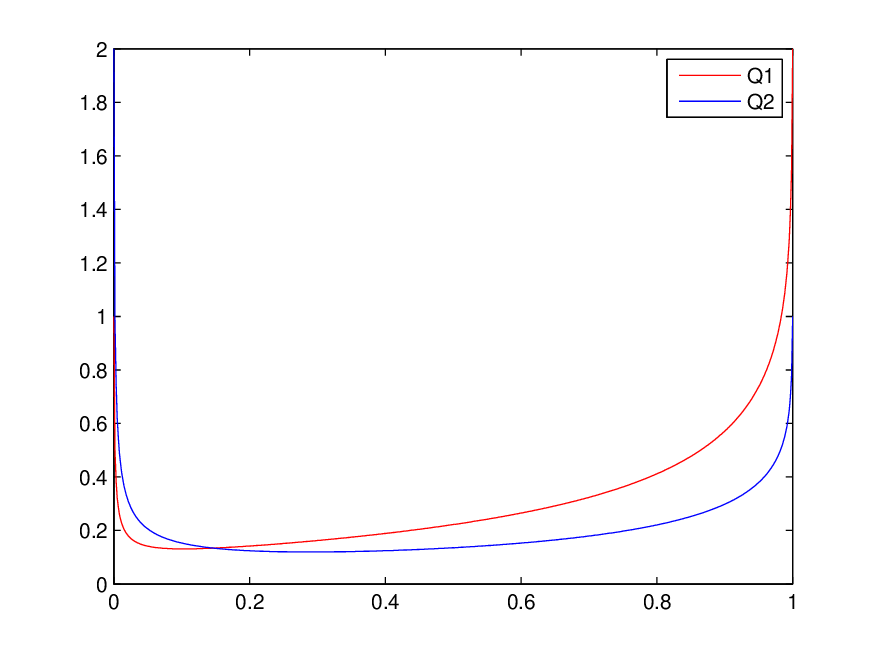}}                
  \subfloat[Quantity $\sum_{r=1}^{N-1} \alpha_r A_r \big( F \cdot \sigma^{r}  (G) \big)$]{\label{constant} \includegraphics[width=0.5\textwidth]{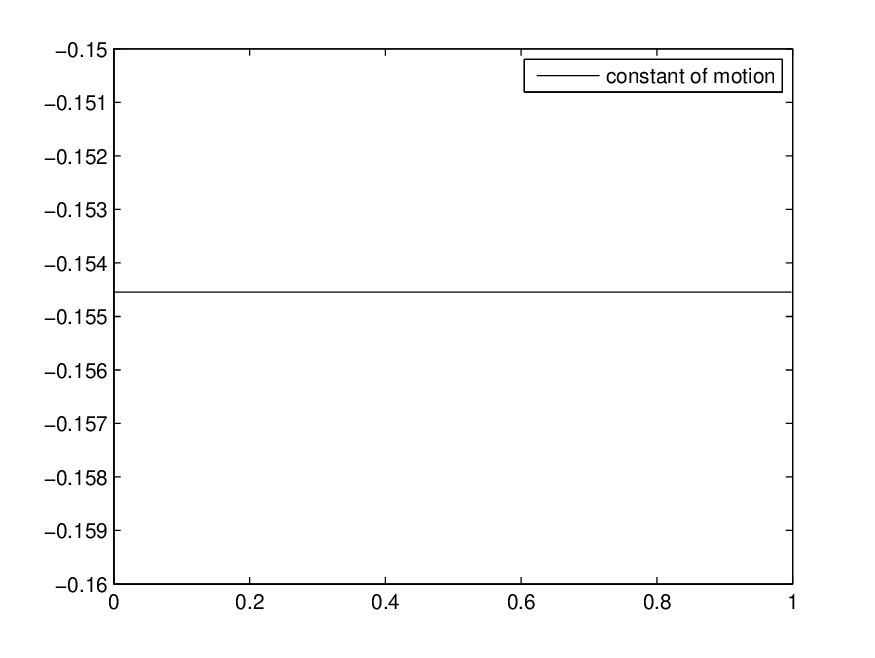}}
  \caption{Computation of \eqref{elfh}}
\end{figure}
\begin{itemize}
\item Figure \ref{Q1Q2} represents the discrete solutions of \eqref{elfh}
\item and Figure \ref{constant} represents the explicit computable quantity $\sum_{r=1}^{N-1} \alpha_r A_r \big( F \cdot \sigma^{r}  (G) \big)$. As expected from Theorem \ref{thmnoetherdfrac}, this quantity is a constant of motion of \eqref{elfh}.
\end{itemize}

\section{Conclusion} 

In this paper, we have proved a fractional Noether's theorem for fractional Lagrangian systems invariant under a symmetry group both in the continuous and discrete cases. In contrary to previous results in this direction (in the continuous case), it provides an \textit{explicit} conservation law. Moreover, the given formula can be algorithmically implemented and in the discrete case, the conservation law is moreover computable in a finite number of steps. \\

As explained in Section \ref{section141}, symmetries with time transformation can be considered and the transfer formula given by Theorem \ref{thmleibizfrac} can be used in order to obtain an explicit constant of motion. Hence, this study is also available in this case and it will be done in a forthcoming paper.
 
\bibliographystyle{plain}
%\bibliography{../bibliototal}

\begin{thebibliography}{10}

\bibitem{agra}
O.P. Agrawal.
\newblock Formulation of {E}uler-{L}agrange equations for fractional
  variational problems.
\newblock {\em J. Math. Anal. Appl.}, 272(1):368--379, 2002.

\bibitem{alme}
R.~Almeida, A.B. Malinowska, and D.F.M. Torres.
\newblock A fractional calculus of variations for multiple integrals with
  application to vibrating string.
\newblock {\em J. Math. Phys.}, 51(3):033503, 12, 2010.

\bibitem{arno}
V.~I. Arnold.
\newblock {\em Mathematical methods of classical mechanics}, volume~60 of {\em
  Graduate Texts in Mathematics}.
\newblock Springer-Verlag, New York, 1979.

\bibitem{atan}
T.M. Atanackovi{\'c}, S.~Konjik, S.~Pilipovi{\'c}, and S.~Simi{\'c}.
\newblock Variational problems with fractional derivatives: invariance
  conditions and {N}oether's theorem.
\newblock {\em Nonlinear Anal.}, 71(5-6):1504--1517, 2009.

\bibitem{bagl}
R.~L. Bagley and R.~A. Calico.
\newblock Fractional order state equations for the control of viscoelastically
  damped structures.
\newblock {\em Journal of Guidance, Control, and Dynamics}, 14:304--311, 1991.

\bibitem{bale2}
D.~Baleanu and S.I. Muslih.
\newblock Lagrangian formulation of classical fields within
  {R}iemann-{L}iouville fractional derivatives.
\newblock {\em Phys. Scripta}, 72(2-3):119--121, 2005.

\bibitem{boni}
B.~Bonilla, M.~Rivero, L.~Rodr{\'{\i}}guez-Germ{\'a}, and J.~J. Trujillo.
\newblock Fractional differential equations as alternative models to nonlinear
  differential equations.
\newblock {\em Appl. Math. Comput.}, 187(1):79--88, 2007.

\bibitem{bour}
L.~Bourdin, J.~Cresson, I.~Greff, and P.~Inizan.
\newblock Variational integrators on fractional {L}agrangian systems in the
  framework of discrete embeddings.
\newblock {\em preprint arXiv:1103.0465v1 [math.DS]}.

\bibitem{comt}
F.~Comte.
\newblock Op\'erateurs fractionnaires en \'econom\'etrie et en finance.
\newblock {\em Pr\'epublication MAP5}, 2001.

\bibitem{cres6}
J.~Cresson.
\newblock Fractional embedding of differential operators and {L}agrangian
  systems.
\newblock {\em J. Math. Phys.}, 48(3):033504, 34, 2007.

\bibitem{torr3}
G.S.F. Frederico and D.F.M. Torres.
\newblock A formulation of {N}oether's theorem for fractional problems of the
  calculus of variations.
\newblock {\em J. Math. Anal. Appl.}, 334(2):834--846, 2007.

\bibitem{lubi}
E.~Hairer, C.~Lubich, and G.~Wanner.
\newblock {\em Geometric numerical integration}, volume~31 of {\em Springer
  Series in Computational Mathematics}.
\newblock Springer-Verlag, Berlin, second edition, 2006.
\newblock Structure-preserving algorithms for ordinary differential equations.

\bibitem{hilf3}
R.~Hilfer.
\newblock Applications of fractional calculus in physics.
\newblock {\em World Scientific, River Edge, New Jersey}, 2000.

\bibitem{kilb}
A.A. Kilbas, H.M. Srivastava, and J.J. Trujillo.
\newblock {\em Theory and applications of fractional differential equations},
  volume 204 of {\em North-Holland Mathematics Studies}.
\newblock Elsevier Science B.V., Amsterdam, 2006.

\bibitem{mach}
J.~Tenreiro Machado, Virginia Kiryakova, and Francesco Mainardi.
\newblock Recent history of fractional calculus.
\newblock {\em Commun. Nonlinear Sci. Numer. Simul.}, 16(3):1140--1153, 2011.

\bibitem{magi}
R.L. Magin.
\newblock Fractional calculus models of complex dynamics in biological tissues.
\newblock {\em Comput. Math. Appl.}, 59(5):1586--1593, 2010.

\bibitem{mars}
J.E. Marsden and M.~West.
\newblock Discrete mechanics and variational integrators.
\newblock {\em Acta Numer.}, 10:357--514, 2001.

\bibitem{musl}
S.~Muslih.
\newblock A formulation of noether's theorem for fractional classical fields.
\newblock {\em preprint arXiv:1003.0653v1 [math-ph]}.

\bibitem{oldh}
K.B. Oldham and J.~Spanier.
\newblock {\em The fractional calculus}.
\newblock Academic Press [A subsidiary of Harcourt Brace Jovanovich,
  Publishers], New York-London, 1974.
\newblock Theory and applications of differentiation and integration to
  arbitrary order, With an annotated chronological bibliography by Bertram
  Ross, Mathematics in Science and Engineering, Vol. 111.

\bibitem{podl}
I.~Podlubny.
\newblock {\em Fractional differential equations}, volume 198 of {\em
  Mathematics in Science and Engineering}.
\newblock Academic Press Inc., San Diego, CA, 1999.
\newblock An introduction to fractional derivatives, fractional differential
  equations, to methods of their solution and some of their applications.

\bibitem{riew}
F.~Riewe.
\newblock Mechanics with fractional derivatives.
\newblock {\em Phys. Rev. E (3)}, 55(3, part B):3581--3592, 1997.

\bibitem{samk}
S.G. Samko, A.A. Kilbas, and O.I. Marichev.
\newblock {\em Fractional integrals and derivatives}.
\newblock Gordon and Breach Science Publishers, Yverdon, 1993.
\newblock Theory and applications, Translated from the 1987 Russian original.

\bibitem{neel}
A.~Zoia, M.-C. N\'eel, and A.~Cortis.
\newblock Continuous-time random-walk model of transport in variably saturated
  heterogeneous porous media.
\newblock {\em Phys. Rev. E}, 81(3):031104, Mar 2010.

\end{thebibliography}

\end{document}